\documentclass[12pt,reqno]{amsart}

\usepackage{amsmath,amsthm,amssymb,comment,fullpage}
\usepackage{braket}
\usepackage{mathtools}
\usepackage{enumitem}

\usepackage{caption}
\usepackage{times}
\usepackage[T1]{fontenc}
\usepackage{mathrsfs}
\usepackage{latexsym}
\usepackage[dvips]{graphics}
\usepackage{epsfig}
\usepackage{amsmath,amsfonts,amsthm,amssymb,amscd}
\input amssym.def
\input amssym.tex
\usepackage{color}
\usepackage{hyperref}
\usepackage{url}
\newcommand{\bburl}[1]{\textcolor{blue}{\url{#1}}}

\usepackage{pgfplots}
\usepackage{pgfplotstable}
\pgfplotsset{compat=1.18}
\usepackage{tikz}
\usepackage{tkz-tab}
\usepackage{tkz-graph}
\usepackage[frozencache,cachedir=.]{minted} 
\usepackage{xcolor}
\usetikzlibrary{shapes.geometric,positioning}

\usepackage[
backend=biber,
style=alphabetic,
sorting=nty
]{biblatex}
\usepackage{dcolumn}
\newcolumntype{d}[1]{D{.}{.}{#1}}

\newtheorem{thm}{Theorem}[section]

\theoremstyle{definition}

\theoremstyle{plain}

\newtheorem{definition}[thm]{Definition}
\newtheorem{example}[thm]{Example}
\newtheorem{lemma}[thm]{Lemma}
\newtheorem{proposition}[thm]{Proposition}
\newtheorem{theorem}[thm]{Theorem}

\numberwithin{equation}{section}

\definecolor{LightGray}{gray}{0.9}

\title{On a Roll Again: Analysis of a Dice Removal Game}

\author{Francesco Camellini$^1$}
\address{$^1$}
\email{\bburl{francesco.camellini@mail.polimi.it}}
\address{Politecnico di Milano, Milan, 20133}

\author{Wissam Ghantous$^2$}
\address{$^2$}
\email{\bburl{wissam.ghantous@ucf.edu}}
\address{Department of Mathematics, University of Central Florida, Orlando, Florida}

\author{Andrea M. Lanocita$^3$}
\address{$^3$}
\email{\bburl{andreamaria.lanocita@mail.polimi.it}}
\address{Politecnico di Milano, Milan, 20133}

\author{Layna E. Mangiapanello$^4$}
\address{$^4$}
\email{\bburl{lm943s@missouristate.edu}}
\address{Missouri State University, Springfield, Missouri, 65807}

\author{Steven J. Miller$^5$}
\address{$^5$}
\email{\bburl{sjm1@williams.edu}}
\address{Department of Mathematics, Williams College, Williamstown, MA 01267, USA}

\author[G. Tresch]{Garrett Tresch$^6$}
\address{$^6$}
\email{\textcolor{blue}{\href{mailto:treschgd@tamu.edu}{treschgd@tamu.edu}}}
\address{Department of Mathematics\\ Texas A\&M University, College Station, TX 77843, USA}

\author{Elif Z. Yildirim$^7$}
\address{$^7$}
\email{\bburl{zumra.yildirim@ug.bilkent.edu.tr}}
\address{Bilkent University, Ankara, 06800}

\thanks{We thank all of Polymath Jr's staff, whose work has been key to the writing of this paper. In particular, we extend our gratitude to Chris Yao. We also thank the referee who reviewed our submission to the PUMP Journal of Undergraduate Research for numerous suggestions that greatly improved this manuscript. Lastly, we thank the editorial staff at the PUMP Journal of Undergraduate Research for their work on the journal. This project was funded by NSF Grant DMS2341670.}

\keywords{Probability Theory, Order Statistics, Maximum of Geometric Sequence. Math Codes: 62G30.}

\date{\today}

\begin{document}

    \begin{abstract} 
    Suppose we have $n$ dice, each with $s$ faces (assume $s\geq n$). On the first turn, roll all of them, and remove from play those that rolled an $n$. Roll all of the remaining dice. In general, if at a certain turn you are left with $k$ dice, roll all of them and remove from play those that rolled a $k$. The game ends when you are left with no dice to roll. 
    For $n,s \in \mathbb{N} \setminus \{0\}$ such that $s \geq n$, let $Y_n^s$ be the random variable for the number of turns to finish the game rolling $n$ dice with $s$ faces.
    We find recursive and non-recursive solutions for $\mathbb{E}(Y_n^{s})$ and $\mathrm{Var}(Y_n^{s})$, and bounds for both values. Moreover, we show that $Y_n^{s}$ can also be modeled as the maximum of a sequence of i.i.d. geometrically distributed random variables. 
    Although, as far as we know, this game hasn't been studied before, similar problems have. The fixed strategy approach to the TENZI dice game analyzed in \cite{veatch2021analysis} is a particular case of our probabilistic model. Other problems involving trials and successes have been explored extensively, such as the Coupon Collector's Problem, discussed in \cite{isaac1996pleasures}.
    \end{abstract}

    \maketitle

    \tableofcontents

    %%%%%%%%%%%%%%%%%%%%%%%%%%%%%%%%%%%%%%%%%%%%%%%%%%%%%%%%%%%%%%%%%%%%%	
	\section{Introduction}\label{s:1}
    As explained in \cite{907cee13-a93d-3906-a37a-8cab426b7d72}, dice have been around since around 3000 BCE. Similarly, there are records of dice games dating as far back as the Roman era, when soldiers had the right to play dice in camp. One of the most popular games involved throwing three dice and summing the numbers shown. This is thought to be the first game studied using mathematical calculations, although the text that illustrates them, called "De Vetula", dates back to the 13th century. Of course, the aim wasn't scientific per se, but moral: the author wanted to describe the game, and thus inform the reader on how he could be ruined. In the following centuries, dice games began to be used to introduce or study the theory of probability. For example, Jakob Bernoulli in \cite{bernoulli1713ars} uses a convolution formula to calculate the chances of the aforementioned game with an arbitrary number of dice, although he only presents the results and not the formula itself, which is supposed to be derived by the reader. Nowadays, the primary aim of studying these games is educational: they are taught as a fun way to present mathematical concepts to students. However, they can still be very valuable in mathematical research. For example, in this paper, using a concrete example, such as the die game, as a starting point helps to get a sense of the equivalence between two formulas to calculate the expected value and variance of the maximum of i.i.d. geometric random variables, which is very complicated to demonstrate mathematically. \\

    The game we study works as follows: Suppose we have $n$ dice, each with $s$ faces (assume $s\geq n$). On the first turn, roll all of them, and remove from play those that rolled an $n$. Roll all of the remaining dice. In general, if at a certain turn you are left with $k$ dice, roll all of them and remove from play those that rolled a $k$. The game ends when you are left with no dice to roll. \\

    As an example, let's simulate a simple game.
    \begin{example} \label{ex:4,6}
        Suppose we have 4 dice, each with 6 faces. We roll all of them on the first turn: $3224$. We remove from the game the fourth die. On the second turn, roll all three remaining dice: $363$. Now we remove from the game the first and third dice. On the third turn, roll the remaining die: $4$. We roll that die again on the fourth turn: $1$. The game ends in four turns.
    \end{example}

    Although, as far as we know, this game hasn't been studied before, similar problems have. The fixed strategy approach to TENZI analyzed in \cite{veatch2021analysis} has a similar solution to our problem. TENZI is a game where each player is given 10 dice with 6 faces and can roll any subset of them. The player wins when all dice show the same number. The fixed strategy for this game is based on choosing a number between 1 and 6 before playing. During the game, you stop rolling forever all of the dice that show the chosen number, while you keep rolling the others. Under these conditions, we can apply to this game the same probabilistic model we used for ours, provided we restrict it to 6 faces and 10 dice (i.e., our model is more general). In fact, that is exactly what is done in \cite{veatch2021analysis}. However, their analysis only provides formulas for the expected value, while ours also finds formulas and bounds for the variance. Furthermore, as said above, we don't restrict our analysis to a particular number of dice or faces for each die.
    
    Other problems involving successes and trials have been explored extensively, such as the Coupon Collector's Problem discussed in \cite{isaac1996pleasures}.  \\

    For $n,s \in \mathbb{N} \setminus \{0\}$ such that $s \geq n$, let $Y_n^s$ be the random variable for the number of turns to finish the game rolling $n$ dice with $s$ faces.
    We model our problem in two different ways. 
    First, we focus on each die individually: we imagine rolling it separately at each turn until we achieve a success and remove it from the game. Note that, although the number of dice can change each turn, the probability of a die achieving a successful roll and being removed from the game is the same. This method results in the number of turns to finish the game being modeled as the maximum of a sequence of i.i.d. geometric random variables, which has been studied before in \cite{SR90, EISENBERG2008135, veatch2021analysis}. These papers include some of our results (namely Theorems \ref{thm:nonRecEV}, \ref{thm:NonRecVar}, \ref{thm:EVRecFormula}, \ref{thm:VarRecForm}).
    Then, we imagine rolling all the dice simultaneously, leading to some successes (i.e., dice that roll the correct number and can be removed from the game) and some failures. This interpretation naturally leads to a Markov-chain interpretation and recursive solutions. 
    For the sake of completeness, we also prove that the formulas found are mathematically equivalent in Appendix B. 
    Furthermore, we find that $n \leq s \leq \mathbb{E}(Y_n^s) \leq ns$ and $s^2-s \leq \mathrm{Var}(Y_n^s) \leq ns(s-1)$. We conclude that the expected value grows linearly in both the number of dice and the number of faces, while the variance grows quadratically in the number of faces and sublinearly in the number of dice. Although the problem of approximating these values had been treated before in \cite{SR90, EISENBERG2008135}, these papers focus on asymptotic results, which can be unreliable for small values of $n$ and $s$. \\

    In the second section, we write some initial definitions and derive the finite number of possible endings. In the third and fourth sections, we find two different formulations for the expected value and variance of the number of turns needed to finish the game. In particular, in the third section, we imagine that each die is rolled individually until a successful roll is achieved (which would remove it from the game). Thus, the problem is equivalent to studying the maximum of a sequence of i.i.d. geometric random variables, and we derive the non-recursive formulas in Theorems \ref{thm:nonRecEV} and \ref{thm:NonRecVar}. In the fourth section, we imagine that, at each turn, all of the dice are rolled simultaneously, leading to a Markov chain interpretation and the recursive formulas in Theorems \ref{thm:EVRecFormula} and \ref{thm:VarRecForm}. In the fifth section, we prove some bounds for both the expected value and variance of the game. \\

    We conclude by noting applications of the formulas and bounds we found. Obviously, these can be applied to any order statistics scenario involving a maximum of i.i.d. geometric random variables, such as networks (\cite{larisch2017crlite, avis2023analysis}), probabilistic data structures in informatics (\cite{benomar2024learning}). Another popular application is bioinformatics (\cite[sec. 5.4, 6.3]{article}), where it is sometimes needed to compare long strands of DNA by counting the length of nucleotide matches. In this setting, the results we prove can be used for P-value approximations. In fact, assuming randomness in the DNA composition, the problem of finding the first non-match between two strands can be modeled as a geometric distribution with parameter $p=0.75$. Therefore, comparing many different strands of DNA, under the null hypothesis of randomness, can be reduced to finding the maximum of IID geometric random variables. 
    
%%%%%%%%%%%%%%%%%%%%%%%%%%%%%%%%%%%%%%%%%%%%%%%%%%%%%%%%%%%%%%%%%%%%%%%%
	\section{Initial Considerations} \label{s:2}
    To study the game, we first define formally what we mean by the term “turn”.
    \begin{definition} \label{def:turns}
        Given $k$ dice, each with $s$ faces, a turn is formally defined as a finite sequence of $k$ numbers, each between $1$ and $s$.
    \end{definition}
    We will, somewhat improperly, also refer to the act of throwing the $k$ dice as “turn”. On each turn, we will colloquially define a die (or a roll) as successful if it shows $k$ (and is therefore removed from the game). \\
    Now we define a game.
    \begin{definition} \label{def:game}
        Given $n$ dice with $s$ faces, a game is formally defined as a possibly infinite sequence of turns, each with $k \ \leq \ n$ dice with $s$ faces, where if we roll $k$ dice on a turn we remove all showing a $k$, and the game ends when there are no dice to roll. 
    \end{definition} 
    We must consider the case where there's an infinite number of turns because we can show that the game doesn't have to finish, though with probability 1 it will.
    \begin{example}
        Let us consider a game with 4 dice, each with 6 faces. In this case, we could have 4 dice forever, if at each throw we never get a 4. This shows that the game is not guaranteed to end.
    \end{example}

    \begin{definition} \label{def:finitegame}
        Given a game with $n$ dice, each with $s$ faces, it is said to be finite if it's a finite sequence of turns. In that case, we define the signature of a game as a sequence of $n$ numbers, one for every different die, each being the last number rolled when the die was removed from the game.
    \end{definition}

    \begin{example}
        Let us consider Example \ref{ex:4,6}. The signature of that game is $4331$.
    \end{example}
    
    \begin{example}
        Let us consider the previous problem again. We observe that the only possible signatures are:
        \[4444, 4441, 4422, 4421, 4333, 4331, 4322, 4321.\]
    \end{example}
    
    \begin{proposition} \label{prop:GameEndings}
        Given a finite game with $n$ dice, each with $s$ faces, there are $2^{n-1}$ possible signatures.
    \end{proposition}
    
    \begin{proof}
        We proceed by induction. With a single die, the game ends only if a 1 is rolled. Thus, there is exactly one possible outcome and the base case holds. \medskip

        Assume that for all integers $k$ such that $1 < k < n$, the number of possible outcomes with $k$ dice is $2^{k-1}$.
        Now consider the case with $n$ dice. Each outcome falls into one of the following categories.
        \begin{itemize}
          \item All $n$ dice show $n$. This accounts for one outcome.
           \item For each $1 \leq i \leq n - 1$, the first $i$ dice show $n$, followed by an outcome of a game with $n - i$ dice.
        \end{itemize}
        Therefore, the total number of outcomes is
        \begin{equation*}
            1 + \sum_{i=1}^{n-1} 2^{i-1}\ = \ 2^{n-1}.
        \end{equation*}
    \end{proof}
    %%%%%%%%%%%%%%%%%%%%%%%%%%%%%%%%%%%%%%%%%%%%%%%%%%%%%%%%%%%%%%%%%

    \section{Subsequent Throws} \label{s:4}
	
	\subsection{Formalization}\label{ss:4.1}
	Now we imagine that each die is rolled singly until achieving a successful roll (which would remove it from the game). 
    The probability that a die with $s$ faces, considered individually, is eliminated at each turn is constant and equal to $p=1/s$. This probability is independent of the outcomes of the other dice or the number of dice being rolled (since the faces are equiprobable). Thus, if we define $Z_i^s$ for $i=1,2,\dots,n$ the random variable for the turn in which die $i$ with $s$ faces is removed from the game,
	we observe that each ${Z_i^s}$ follows a geometric distribution with parameter $p$.
	Therefore, if, as before, we define $Y_n^s$ as the random variable for the number of turns to finish the game rolling $n$ dice with $s$ faces, we note that $Y_n^{s}$ corresponds to the turn at which the last die is eliminated, and thus
	\begin{equation*}
		Y_n^{s} \ = \  \max_i {Z_i^s}.
	\end{equation*}

    \begin{proposition} \label{NonRecDistrib}
        Assume $s\geq n>0$, $\ {Z_i^s}\overset{\mathrm{iid}}{\sim}\mathrm{Geom}(1/s)$, $\ Y_n^{s}=\max_{i\leq n}{Z_i^s}$, and $q=1-p$. Then

        \begin{equation}\label{eq:nonRecDistrib} 
            \mathbb{P}(Y_n^{s}=y) \ = \ \sum_{k=1}^n \binom{n}{k}(-1)^{k+1}q^{yk}\left(\frac{1}{q^k}-1\right).
        \end{equation}
    \end{proposition}

        \begin{proof}
        For $y=1,2,\dots$, we calculate the cumulative distribution function using some well-known results in order statistics \cite[p.229]{casella2002statistical}:
	\begin{equation} \label{eq:nonrecCumdistrib}
		F_{Y_n^{s}}(y)  \ = \ F_{\max_i {Z_i^s}}(y) \ = \ F_{{Z_i^s}}(y)^n \ = \
		\begin{cases}
			0 &y<1;\\
			\left(1-(1-p)^{y}\right)^n \qquad &y\geq 1.
		\end{cases}
	\end{equation}
	From this, we can also calculate the probability density,
	\begin{align*} 
		\mathbb{P}(Y_n^{s}=y) \ &= \ F_{Y_n^{s}}(y)-F_{Y_n^{s}}(y-1)\\ 
		\ &= \ \left(1-(1-p)^{y}\right)^n-\left(1-(1-p)^{y-1}\right)^n\\ 
		\ &= \ \sum_{k=0}^{n} \binom{n}{k} (-1)^{k}(1-p)^{yk} - \sum_{k=0}^n \binom{n}{k}(-1)^{k}(1-p)^{(y-1)k}\\
		\ &= \ \sum_{k=0}^n \binom{n}{k}(-1)^{k}\left((1-p)^{yk}-(1-p)^{(y-1)k}\right).
    \end{align*}
    Then, if we let $q=1-p$ the above becomes
    \begin{align*} 
        \mathbb{P}(Y_n^{s}=y) \ &= \ \sum_{k=1}^n \binom{n}{k}(-1)^kq^{yk}\left(1-\frac{1}{q^k}\right)\\
		\ &= \ \sum_{k=1}^n \binom{n}{k}(-1)^{k+1}q^{yk}\left(\frac{1}{q^k}-1\right).
	\end{align*}
    \end{proof}
	
	\subsection{Expected Value}\label{ss:4.2}
	Now we can use Equation \eqref{eq:nonRecDistrib} to find a new equation for $\mathbb{E}(Y_n^s)$.

    \begin{theorem} \label{thm:nonRecEV}
        Assume $s\geq n>0$, $\ {Z_i^s}\overset{\mathrm{iid}}{\sim}\mathrm{Geom}(1/s)$, $\ Y_n^{s}=\max_{i\leq n}{Z_i^s}$, and $q=1-p$. Then
        \begin{equation}\label{eq:EvnonRec}
            \mathbb{E}(Y_n^{s}) \ = \ \sum_{k=1}^n\binom{n}{k}(-1)^{k+1}\frac{1}{1-q^k}.  
        \end{equation}
    \end{theorem}

    \begin{proof}
      The expected value is
    \begin{align*} 
        \mathbb{E}(Y_n^{s}) \ &= \ \sum_{y=1}^{\infty}y \sum_{k=1}^n \binom{n}{k}(-1)^{k+1}q^{yk}\left(\frac{1}{q^k}-1\right)\\
        \ &= \ \sum_{k=1}^n\binom{n}{k}(-1)^{k+1}\left(\frac{1}{q^k}-1\right) \sum_{y=1}^{\infty}y \cdot q^{yk}.
    \end{align*}
    By Lemma \ref{lemma:series}, when $q^k<1$, we have $\sum_{y=1}^{\infty}y \cdot q^{yk}=\frac{q^k}{(1-q^k)^2}$. Therefore
    \begin{align*}
        \mathbb{E}(Y_n^{s}) \ &= \ \sum_{k=1}^n\binom{n}{k}(-1)^{k+1}\left(\frac{1}{q^k}-1\right)\frac{q^k}{(1-q^k)^2}\\
        \ &= \ \sum_{k=1}^n\binom{n}{k}(-1)^{k+1}\frac{1}{1-q^k}.
    \end{align*}  
    \end{proof}

    Note that, as a direct consequence of Theorem \ref{thm:nonRecEV}, we have $\mathbb{E}(Y_n^s)<\infty$.
	
    \subsection{Variance} \label{ss:4.3}
    Now we can use Equation \eqref{eq:nonRecDistrib} to find a new equation for $\mathrm{Var}(Y_n^s)$. 

    \begin{theorem} \label{thm:NonRecVar}
        Assume $s \geq n>0$, $\ {Z_i^s}\overset{\mathrm{iid}}{\sim}\mathrm{Geom}(1/s)$, $\ Y_n^{s}=\max_{i\leq n}{Z_i^s}$, and $q=1-p$. Then
        \begin{equation} \label{eq:VarnonRec}
            \mathrm{Var}(Y_n^{s}) \ = \ \sum_{k=1}^n\binom{n}{k}(-1)^{k+1}\frac{1+q^k}{(1-q^k)^2} - \mathbb{E}(Y_n^{s})^2.   
        \end{equation}
    \end{theorem}

    \begin{proof}
    The second moment is
    \begin{align*}
        \mathbb{E}\left(\left({Y_n^{s}}\right)^2\right) \ &= \ \sum_{y=1}^{\infty}y^2 \sum_{k=1}^n \binom{n}{k}(-1)^{k+1}q^{yk}\left(\frac{1}{q^k}-1\right)\nonumber\\
        \ &= \ \sum_{k=1}^n\binom{n}{k}(-1)^{k+1}\left(\frac{1}{q^k}-1\right) \sum_{y=1}^{\infty}y^2 \cdot q^{yk}\nonumber
    \end{align*}
    By Lemma \ref{lemma:series}, when $q^k<1$, we have $\sum_{y=1}^{\infty}y^2 \cdot q^{yk}=\frac{q^k(1+q^k)}{(1-q^k)^3}$. Therefore:
    \begin{align}\label{eq:EV2NonRec}
        \mathbb{E}\left(\left({Y_n^{s}}\right)^2\right)\ &= \ \sum_{k=1}^n\binom{n}{k}(-1)^{k+1}\left(\frac{1}{q^k}-1\right)\frac{q^k(1+q^k)}{(1-q^k)^3}\nonumber\\
        \ &= \ \sum_{k=1}^n\binom{n}{k}(-1)^{k+1}\frac{1+q^k}{(1-q^k)^2}.
    \end{align} 
    The variance can then be found using the well-known formula $\mathrm{Var}(Y_n^{s})=\mathbb{E}\left(\left({Y_n^{s}}\right)^2\right)-\mathbb{E}(Y_n^{s})^2$.
    \end{proof}

    Note that, as a direct consequence of Theorem \ref{thm:NonRecVar}, we have $\mathrm{Var}(Y_n^s)<\infty$.
    
%%%%%%%%%%%%%%%%%%%%%%%%%%%%%%%%%%%%%%%%%%%%%%%%%%%%%%%%%%%%%%%%%%%%%%%
    
	\section{Markovian Approach}  \label{s:3}
	
	\subsection{Formalization}\label{ss:3.1}
	We consider all the dice to be rolled simultaneously. Our system is described by the number of dice present at turn $t$. Given an initial number $n$ of dice (with $s$ faces each), during the game we can only have $m \in \mathbb{N}$ dice with $0 \leq m \leq n$. Hence, we can define the state space $S = \{0, 1, \dots , n\}$ and the stochastic process $(X_t^s)_{t \geq 0}$, where $X^s_t\in S$ denotes the number of dice at turn $t$, with initial condition $X_0^s=n$. We observe that, given a state $X^s_t$, the distribution of $X^s_{t+1}$ depends exclusively on $X^s_t$. Therefore, our process is a discrete-time Markov chain.\\Since each die is rolled independently and each roll results either in a success (die is eliminated) or a failure (die remains), the number of successful rolls at a turn with $k$ dice (having $s$ faces each) follows a binomial distribution
    \[
    {B^{s}_{k}} \ \sim \ \text{Bin}\left(k, p\right), \quad\text{with}\quad p=1/s.
    \]
    We now define the transition probabilities. 
    \begin{align} \label{eq:transProb} 
    \notag
        P_{i, j} \ = \ \mathbb{P}(X^s_{t+1} = j \ | \ X^s_t = i) \ &= \ \begin{cases}
            \mathbb{P}(B_i^s = i-j) \quad &\text{if } 0 \leq j \leq i;\\ 
            0 \quad &\text{otherwise}.\\ 
        \end{cases} 
        \\ \ &= \ \begin{cases}
            \binom{i}{j}p^{i-j}(1-p)^j \quad &\text{if } 0 \leq j \leq i;\\ 
            0 \quad &\text{otherwise}.\\ 
        \end{cases}
    \end{align}
    Note that $X^s_t=0$ is the only absorbing state of the process and corresponds to the end of the game. In fact, due to Equation \eqref{eq:transProb}, $P_{0, y} = 0, \ \forall y \in S\setminus \{0\}$. Therefore, to study the duration of the game, we consider the absorption time
    \[
    T_n^s \ = \ \inf \{ t \geq 0 : \ X^s_t = 0 \ | \ X^s_0 = n \}.
    \]
    
    \subsection{Expected Value}\label{ss:3.2}

    \begin{theorem}\label{thm:EVRecFormula}
        Assume $s\geq n>0$, $p=1/s$. Then,
    \begin{equation} \label{eq:EV_recurrence}
    \mathbb{E}(T_n^{s}) \ = \ \frac{1+ \sum_{k=0}^{n-1} \binom{n}{k}p^{n-k}(1-p)^k \mathbb{E}(T_k^s)}{1-(1-p)^n}.
    \end{equation}
    \end{theorem}

    \begin{proof}
    Let $\,\tau = \inf \{ m \geq 0 : \ X^s_m = 0 \}$. We want to compute $\mathbb{E}(T_n^s) =  \mathbb{E}(\tau \ | \ X^s_0 = n)$.\\ 
    We can write $\tau = 1 + \tau'$, where $\tau '$ is the remaining time to absorption after time $1$.
    \begin{equation*}
    \mathbb{E}(T_n^{s}) \ = \ \mathbb{E}(\tau \ | X^s_0 = n) = 1 + \mathbb{E}(\tau ' \ | X^s_0 = n). 
    \end{equation*}
    Now we condition on $X^s_1$ to obtain
    \begin{equation*}
        \mathbb{E}(T_n^{s}) \ = \ 1 + \sum_{k=0}^n \mathbb{E}(\tau' \ | \ X^s_1 = k) \,\mathbb{P}(X^s_1 = k \ | \ X^s_0 = n) \ =  \ 1 + \sum_{k=0}^n P_{n,k}\, \mathbb{E}(\tau' \ | \ X^s_1 = k). 
    \end{equation*}
    Since our Markov chain is time-homogeneous, we have $\mathbb{E}(\tau' \ | \ X^s_1 = k) = \mathbb{E}(\tau \ | \ X^s_0 = k)$ for any $k\ge0$ (\cite[p.446]{Cinlar_2011}). Therefore,
    \begin{align*}
    \mathbb{E}(T_n^{s}) \ = \ 1 + \sum_{k=0}^n P_{n,k}\, \mathbb{E}(T_k^s) \ = \ 1 + P_{n,n}\, \mathbb{E}(T_n^s) + \sum_{k=0}^{n-1} P_{n,k}\, \mathbb{E}(T_k^s) .
    \end{align*}
    From this, we obtain the formula
    \begin{equation*} 
    \mathbb{E}(T_n^{s})\ = \ \frac{\sum_{k=0}^{n-1} P_{n,k} \,\mathbb{E}(T_k^s) +1}{1-P_{n,n}}.
    \end{equation*}
    Lastly, we can use Equation \eqref{eq:transProb} to find
    \begin{equation*}
    \mathbb{E}(T_n^{s})\ = \ \frac{1+ \sum_{k=0}^{n-1} \binom{n}{k}p^{n-k}(1-p)^k \mathbb{E}(T_k^s)}{1-(1-p)^n}.
    \end{equation*}
    \end{proof}
    
\subsection{Variance} \label{ss:3.3}
We compute $\mathrm{Var}(T_n^s)$ using a strategy similar to the one used in Theorem \ref{thm:EVRecFormula}.

    \begin{theorem} \label{thm:VarRecForm}
        Assume $s \geq n>0$, $p=1/s$. Then
        \begin{equation}
        \mathrm{Var}(T_n^{s}) \ = \ \frac{\sum_{k=0}^{n-1} \binom{n}{k}p^{n-k}(1-p)^k\,\mathbb{E}((T_k^s)^2)-1 + 2\mathbb{E}(T_n^s)}{1-(1-p)^n} - \mathbb{E}(T_n^s)^2.
        \end{equation}
    \end{theorem}

    \begin{proof}
    We define $\tau$ and $\tau'$ as in the proof of Theorem \ref{thm:EVRecFormula}, with $\tau = 1 + \tau'$.\\We want to find $\mathbb{E}((T_n^s)^2) =  \mathbb{E}(\tau^2 \ | \ X^s_0 = n)$. We have
    \begin{align*}
        \mathbb{E}((T_n^s)^2) \ &= \ \mathbb{E}(\tau ^ 2 \ | \ X^s_0 = n) \\
        \ &= \ 1 + 2\mathbb{E}(\tau' \ | \ X^s_0 = n) + \mathbb{E}((\tau')^2 \ | X^s_0 = n) \\
        \ &= \ 1+ 2\sum_{k=0}^n P_{n,k}\,\mathbb{E}(\tau' \ | \ X^s_1 = k) + \sum_{k=0}^n P_{n,k}\,\mathbb{E}((\tau') ^2 \ | \ X^s_1 = k) \\
        \ &= \ 1+ 2\sum_{k=0}^n P_{n,k}\,\mathbb{E}(T_k^s) + \sum_{k=0}^n P_{n,k}\,\mathbb{E}((T_k^s)^2).
    \end{align*}
    As seen in Theorem \ref{thm:EVRecFormula}, $\mathbb{E}(T_n^{s}) = 1 + \sum_{k=0}^n P_{n,k}\, \mathbb{E}(T_k^s)$, thus $\sum_{k=0}^n P_{n,k}\, \mathbb{E}(T_k^s) =  \mathbb{E}(T_n^{s}) - 1$. Therefore, we have
    \begin{align*}
        \mathbb{E}((T_n^s)^2) \ &= \ -1 + 2 \mathbb{E}(T_n^s) + \sum_{k=0}^{n-1} P_{n,k}\,\mathbb{E}((T_k^s)^2) + P_{n,n}\,\mathbb{E}((T_n^s)^2).
    \end{align*}
    We solve for $\mathbb{E}((T_n^s)^2)$ and obtain
    \begin{align*}
        \mathbb{E}((T_n^s)^2) \ &= \ \frac{\sum_{k=0}^{n-1} P_{n,k}\,\mathbb{E}((T_k^s)^2) + 2 \mathbb{E}(T_n^s) - 1}{1-P_{n,n}} \ = \ \frac{\sum_{k=0}^{n-1} \binom{n}{k}p^{n-k}(1-p)^k\,\mathbb{E}((T_k^s)^2)-1 + 2\mathbb{E}(T_n^s)}{1-(1-p)^n}.
    \end{align*}
	We can find the variance using the well-known formula $\mathrm{Var}(T_n^{s})=\mathbb{E}(\left({T_n^{s}}\right)^2)-\mathbb{E}\left(T_n^{s}\right)^2$.\\ 
    \end{proof}

%%%%%%%%%%%%%%%%%%%%%%%%%%%%%%%%%%%%%%%%%%%%%%%%%%%%%%%%%%%%%%%%%%%%%%%%%%%%%%%%%%%%%%%%%%

\section{Bounds} \label{s:6}
We establish lower and upper bounds for both the expected value and the variance, using some elementary bounds (Propositions \ref{prop:EVBound1}, \ref{prop:VarBound1}) as a baseline to determine their effectiveness. 

\subsection{Expected Value} \label{ss:6.1}
We approach the problem as we did in Section \ref{s:4}, considering the throws subsequently. We are interested in bounding the value $\mathbb{E}(Y_n^{s})$, where $Y_n^{s}=\max_i{Z_i^s}$, with ${Z_i^s}$ geometrically distributed i.i.d. random variables. The elementary bound we present in Proposition \ref{prop:EVBound1} and  can be further improved by using the exact value of some $\mathbb{E}(Y_n^{s})$ for small values of $n$.

\begin{proposition}\label{prop:EVBound2}
Assume $s \ \geq \ n \ > \ 0$, $\ {Z_i^s}\overset{\mathrm{iid}}{\sim}\mathrm{Geom}(1/s)$, $\ Y_n^{s}=\max_{i\leq n}{Z_i^s}$. Then
\begin{enumerate}
    \item if $n$ is even, $\mathbb{E}(Y_n^{s}) \ \leq \ \frac{n}{2} \frac{3s^2-2s}{2s-1}$, and
    \item if $n$ is odd,  $\mathbb{E}(Y_n^{s}) \ \leq \ \frac{n}{2} \frac{3s^2-2s}{2s-1} + s$.

\end{enumerate}

\end{proposition}

\begin{proof}
    We use Theorem \ref{thm:nonRecEV} to calculate the expected value for $n=2$.
    \begin{equation*}
    \mathbb{E}({Y_{2}^{s}}) \ = \  \frac{2}{1-(1-1/s)}-\frac{1}{1-(1-1/s)^2} \ = \ 2s-\frac{s^2}{2s-1} \ = \ \frac{3s^2-2s}{2s-1}.
    \end{equation*}
    Now, we assume $n$ is even, we define $I=\{1,3,\dots,n-1\}$ and we rewrite $Y_n^{s}$ as
    \begin{equation*}
    Y_n^{s} \ = \ \max(Z_1^s, \dots, Z_n^s) \ = \ \max[\max(Z_1^s, Z_2^s), \dots, \max(Z_{n-1}^s, Z_n^s)] \ \leq \ \sum_{i\,\in\, I}\max({Z_i^s}, Z_{i+1}^s).
    \end{equation*}
    Note that since $\{Z_n^s\}_n$ are i.i.d., then $\max({Z_i^s}, Z_{i+1}^s) \sim {Y_{2}^{s}}$. In particular, they have the same expected value. By the monotonicity and linearity of the expected value,
    \begin{equation} \label{eq:EVBound2}
    \mathbb{E}(Y_n^{s}) \ \leq \ \sum_{i\,\in\, I}\mathbb{E}[\max({Z_i^s}, Z_{i+1}^s)] \ = \ \sum_{i\,\in\, I}\mathbb{E}({Y_{2}^{s}}) \ = \ \frac{n}{2} \frac{3s^2-2s}{2s-1}.
    \end{equation}
    We can extend this formula to an odd number of dice. Let $m=n+1$, with $n$ even. Then
    \begin{align*} 
       Y_m \ &= \ \max(Z_1^s, ..., Z_n^s, Z_{n+1}^s) \\  
       \ &= \ \max[\max(Z_1^s, Z_2^s), ..., \max(Z_{n-1}^s, Z_n^s), Z_{n+1}^s] \\
       \ & \leq \ \sum_{i\,\in\, I}\max({Z_i^s}, Z_{i+1}^s) + Z_{n+1}^s .
    \end{align*}
    Similarly, by the monotonicity and linearity of the expected value,
    \begin{equation}
    \mathbb{E}(Y_m^{s}) \ \leq \ \sum_{i\,\in\, I}\mathbb{E}\left[\max({Z_i^s}, Z_{i+1}^s)\right] + \mathbb{E}(Z_{n+1}^s) \ = \sum_{i\,\in\, I}\mathbb{E}({Y_{2}^{s}}) \, +\mathbb{E}(Z_{n+1}^s) \ = \ \frac{n}{2} \frac{3s^2-2s}{2s-1} + s.
    \end{equation}
\end{proof}

We show numerically how the bound in Proposition \ref{prop:EVBound2} is an improvement over the bound in Proposition \ref{prop:EVBound1} in Figure \ref{fig:EVBound2}.

\begin{figure}[h!]
\centering
\caption{Comparison of bounds in Equations \eqref{eq:EVBound1} vs \eqref{eq:EVBound2}.}
\label{fig:EVBound2}

% ----------- First Chart: EV, s*n, UB vs n -----------
\begin{minipage}{0.48\textwidth}
\centering
\begin{tikzpicture}
\begin{axis}[
    xlabel={$n$},
    title={Fixed $s=10$},
    legend style={at={(0.05,0.95)},anchor=north west},
    grid=both,
    grid style={dashed, gray!30},
    tick align=outside,
    tick style={black},
    xtick={2,4,6,8,10},
    ymajorgrids=true
]

% EV
\addplot[
    color=blue,
    mark=*,
    thick
] coordinates {
    (2, 14.73684211)
    (4, 20.27337819)
    (6, 23.75349288)
    (8, 26.29578437)
    (10, 28.2994867)
};
\addlegendentry{$\mathbb{E}(Y_n^{s})$}

% s*n
\addplot[
    color=red,
    mark=square*,
    thick,
    dotted
] coordinates {
    (2, 20)
    (4, 40)
    (6, 60)
    (8, 80)
    (10, 100)
};
\addlegendentry{$s \cdot n$}

% Upper Bound
\addplot[
    color=green!60!black,
    mark=triangle*,
    thick,
    dotted
] coordinates {
    (2, 14.73684211)
    (4, 29.47368421)
    (6, 44.21052632)
    (8, 58.94736842)
    (10, 73.68421053)
};
\addlegendentry{$\frac{n}{2} \frac{3s^2-2s}{2s-1}$}

\end{axis}
\end{tikzpicture}
\end{minipage}
\hfill
% ----------- Second Chart: EV, s*n, UB vs S -----------
\begin{minipage}{0.48\textwidth}
\centering
\begin{tikzpicture}
\begin{axis}[
    xlabel={$s$},
    title={Fixed $n=4$},
    legend style={at={(0.05,0.95)},anchor=north west},
    grid=both,
    grid style={dashed, gray!30},
    tick align=outside,
    tick style={black},
    xtick={4,6,8,10,12},
    ymajorgrids=true
]

% EV
\addplot[
    color=blue,
    mark=*,
    thick
] coordinates {
    (4, 7.74178)
    (6, 11.9267)
    (8, 16.1018)
    (10, 20.2734)
    (12, 24.4433)
};
\addlegendentry{$\mathbb{E}(Y_n^{s})$}

% s*n
\addplot[
    color=red,
    mark=square*,
    thick,
    dotted
] coordinates {
    (4, 16)
    (6, 24)
    (8, 32)
    (10, 40)
    (12, 48)
};
\addlegendentry{$s \cdot n$}

% Upper Bound
\addplot[
    color=green!60!black,
    mark=triangle*,
    thick,
    dotted
] coordinates {
    (4, 11.42857143)
    (6, 17.45454545)
    (8, 23.46666667)
    (10, 29.47368421)
    (12, 35.47826087)
};
\addlegendentry{$\frac{n}{2} \frac{3s^2-2s}{2s-1}$}

\end{axis}
\end{tikzpicture}
\end{minipage}
\end{figure}
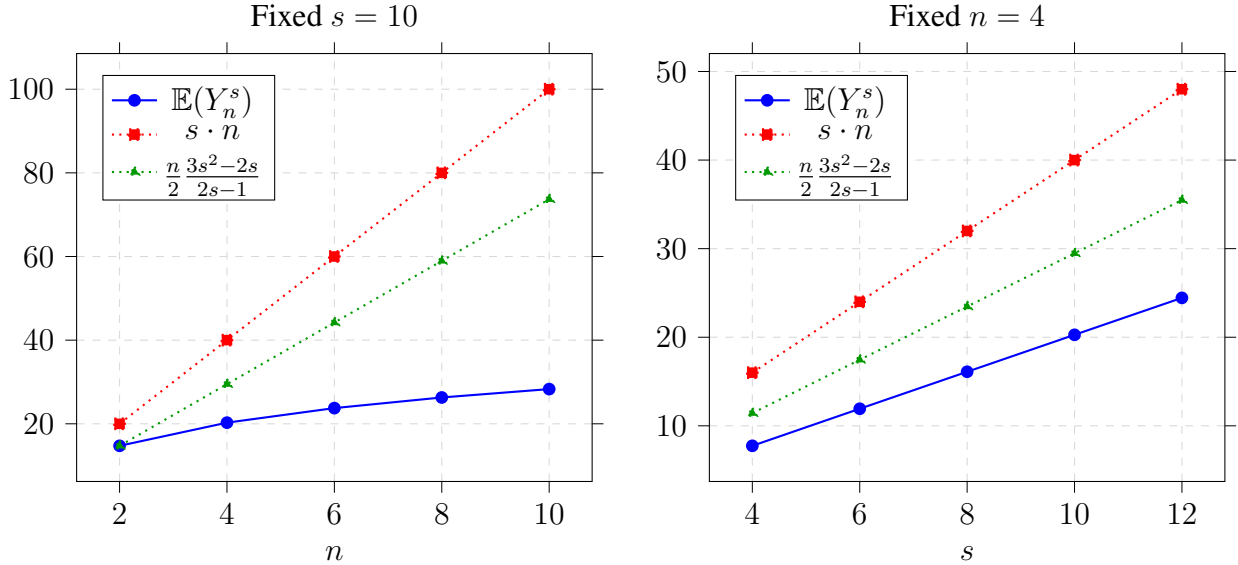

\subsection{Variance}
Once again, we consider the rolls as subsequent.\\

We present a lemma that expresses a fundamental property of associated random variables. Its proof can be found in \cite{AssRandomVar}. 

\begin{lemma}\label{lemma:covariance}
    Let $X = (X_1, \dots, X_{n}^{s})$ be a vector of independent random variables. Then for any real-valued functions $f$ and $g$ that are non-decreasing in each argument,
    \begin{equation*}
        \mathrm{Cov}(f(X), g(X)) \ \geq \ 0.
    \end{equation*}
\end{lemma}

\begin{lemma}\label{lemma:VarMax<Sum}
    Let $\{{Z_i^s}\}^n_{i=1}$ be a sequence of independent random variables. Then
    \begin{equation}
    \mathrm{Var}\left(\max_{i\leq n}({Z_i^s})\right) \ \leq \ \sum_{i=1}^n\mathrm{Var}({Z_i^s}).
    \end{equation}
\end{lemma}

\begin{proof}
Let $\Gamma = \ \{1,..,n\} \ \setminus \ \{\arg\, \max_{i\leq n}{Z_i^s}\}$. By the independence of the random variables, we have
\begin{align*} 
    \sum_{i=1}^n\mathrm{Var}({Z_i^s}) \ &= \ \mathrm{Var}\left(\sum_{i=1}^n{Z_i^s}\right) \\  \ &= \ \mathrm{Var}\left(\max_{i\leq n}{Z_i^s} + \sum_{i \in \Gamma}{Z_i^s}\right) \\  \ &= \ \mathrm{Var}\left(\max_{i\leq n}{Z_i^s}\right) + \mathrm{Var}\left(\sum_{i\in \Gamma}{Z_i^s}\right) + 2\ \mathrm{Cov}\left(\max_{i\leq n}{Z_i^s},\sum_{i \in \Gamma}{Z_i^s}\right)\\
    &\geq \ \mathrm{Var}\left(\max_{i\leq n}{Z_i^s}\right)
\end{align*}
where the inequality follows from the non-negativity of the covariance term due to Lemma \ref{lemma:covariance}.
\end{proof}

\begin{proposition} \label{prop:VarBound2}
Assume $s \ \geq \ n \ > \ 0$, $\ {Z_i^s}\overset{\mathrm{iid}}{\sim}\mathrm{Geom}(1/s)$, $\ Y_n^{s}=\max_{i\leq n}{Z_i^s}$. Then
\begin{equation} \label{eq:VarBound2}
\mathrm{Var}(Y_n^{s}) \ \leq \ ns(s-1).
\end{equation}
\end{proposition}

\begin{proof}
    The proposition is a direct consequence of Lemma \ref{lemma:VarMax<Sum}:
    \begin{equation*}
    \mathrm{Var}(Y_n^{s}) \ = \ \mathrm{Var}\left(\max_{i\leq n}{Z_i^s}\right) \ \leq \ \sum_{i=1}^n \mathrm{Var}({Z_i^s}) \ = \ n\,\frac{1-1/s}{1/s^2} \ = \ ns(s-1).
    \end{equation*}
\end{proof}

In Figure \ref{fig:VarBound2} we show how this bound is an improvement over the elementary bound in Proposition \ref{prop:VarBound1}.

\begin{figure}[h!]
\centering
\caption{Comparison of bounds in Equations \eqref{eq:VarBound1} vs \eqref{eq:VarBound2}}
\label{fig:VarBound2}

% ----------- First Chart: Variance vs n -----------
\begin{minipage}{0.48\textwidth}
\centering
\begin{tikzpicture}
\begin{axis}[
    xlabel={$n$},
    title={Fixed $s=10$},
    legend style={at={(0.05,0.95)},anchor=north west},
    grid=both,
    grid style={dashed, gray!30},
    tick align=outside,
    tick style={black},
    xtick={2,4,6,8,10},
    ymajorgrids=true
]

% Variance
\addplot[
    color=blue,
    mark=*,
    thick
] coordinates {
    (2, 112.6869806)
    (4, 128.3269068)
    (6, 134.4325467)
    (8, 137.6785326)
    (10, 139.6915048)
};
\addlegendentry{$VarY_n$}

% s^2(2n-1) - ns
\addplot[
    color=green!60!black,
    mark=triangle*,
    thick,
    dotted
] coordinates {
    (2, 280)
    (4, 660)
    (6, 1040)
    (8, 1420)
    (10, 1800)
};
\addlegendentry{$s^2(2n-1)-ns$}

% ns(s-1)
\addplot[
    color=red,
    mark=square*,
    thick,
    dotted
] coordinates {
    (2, 180)
    (4, 360)
    (6, 540)
    (8, 720)
    (10, 900)
};
\addlegendentry{$ns(s-1)$}

\end{axis}
\end{tikzpicture}
\end{minipage}
\hfill
% ----------- Second Chart: Variance vs S -----------
\begin{minipage}{0.48\textwidth}
\centering
\begin{tikzpicture}
\begin{axis}[
    xlabel={$s$},
    title={Fixed $n=2$},
    legend style={at={(0.05,0.95)},anchor=north west},
    grid=both,
    grid style={dashed, gray!30},
    tick align=outside,
    tick style={black},
    xtick={2,4,6,8,10},
    ymajorgrids=true
]

% Variance
\addplot[
    color=blue,
    mark=*,
    thick
] coordinates {
    (2, 2.666666667)
    (4, 15.18367347)
    (6, 37.68595041)
    (8, 70.18666667)
    (10, 112.6869806)
};
\addlegendentry{$VarY_n$}

% s^2(2n-1)
\addplot[
    color=green!60!black,
    mark=triangle*,
    thick,
    dotted
] coordinates {
    (2, 8)
    (4, 40)
    (6, 96)
    (8, 176)
    (10, 280)
};
\addlegendentry{$s^2(2n-1)-ns$}

% ns(s-1)
\addplot[
    color=red,
    mark=square*,
    thick,
    dotted
] coordinates {
    (2, 4)
    (4, 24)
    (6, 60)
    (8, 112)
    (10, 180)
};
\addlegendentry{$ns(s-1)$}

\end{axis}
\end{tikzpicture}
\end{minipage}

\end{figure}
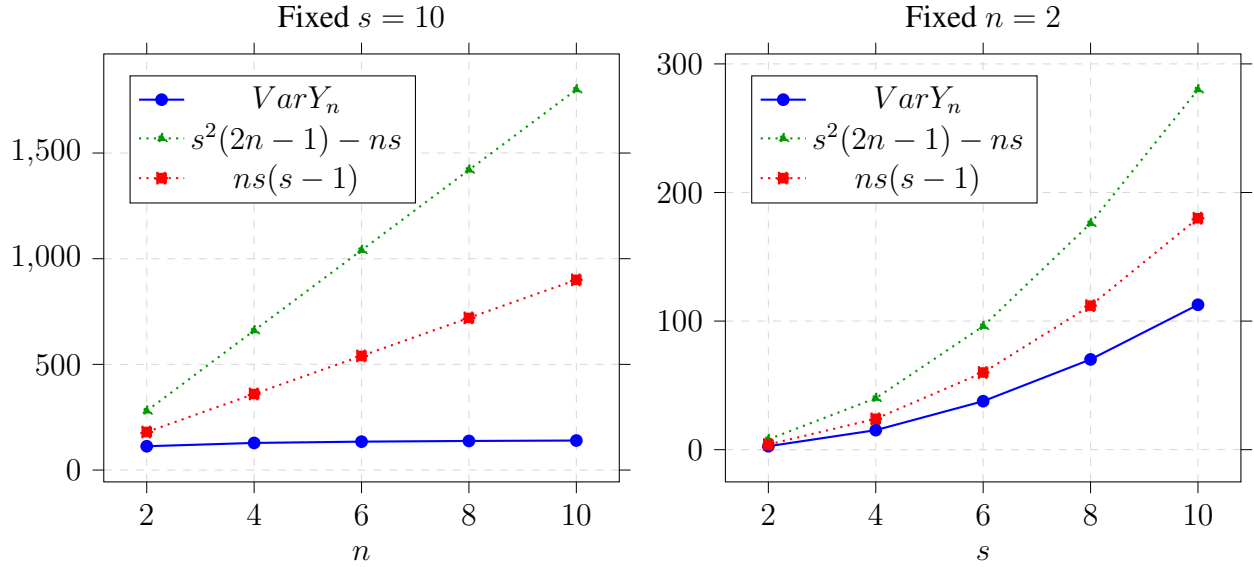

\newpage
\appendix
\section{Additional Results}
In this appendix, for the sake of completeness, we present some simple lemmas used to prove results in the paper.
\begin{lemma}\label{lemma:series}
    Let $a,b \in \mathbb{R}$, $|a^b|<1$. Then
    \begin{equation}
        \sum_{i=1}^{\infty} ia^{bi} \ = \ \frac{a^b}{(1-a^b)^2}, \qquad \sum_{i=1}^{\infty} i^2a^{bi} \ = \ \frac{a^b(1+a^b)}{(1-a^b)^3}.
    \end{equation}
\end{lemma}

\begin{proof}
    Let $x=a^b\in(-1,1)$. By the geometric series formula,
    \begin{equation*}
        \sum_{i=0}^{\infty}x^{i} \ = \ \frac{1}{1-x}.
    \end{equation*}
    For all $x$ in $[-r,r]\subset(-1,1)$ the series $\sum_{i=0}^{\infty}ix^{i-1}$ converges uniformly and so we may differentiate both sides with respect to $x$:
    \begin{equation*}
        \frac{d}{dx}\left(\sum_{i=0}^{\infty}x^{i}\right) \ = \ \sum_{i=0}^{\infty}ix^{i-1} \ = \ \frac{1}{(1-x)^2}.
    \end{equation*}
    Then, multiplying both terms by $x$ yields:
    \begin{equation*}
        \sum_{i=0}^{\infty}ix^{i} \ = \ \frac{x}{(1-x)^2}, \qquad |x|<1.
    \end{equation*}
    Since the $i=0$ term contributes 0, we may drop it. So, substituting back $a^b=x$, we have the desired result.\\
    Now we prove the second result. We know that
    \begin{equation*}
        \sum_{i=0}^{\infty}ix^{i} \ = \ \frac{x}{(1-x)^2}.
    \end{equation*}
    For all $x$ in $[-r,r]\subset(-1,1)$ the series $\sum_{i=0}^{\infty}i^2x^{i-1}$ converges uniformly and so we may differentiate both sides with respect to $x$
    \begin{equation*}
        \frac{d}{dx}\left(\sum_{i=0}^{\infty}ix^{i}\right) \ = \ \sum_{i=0}^{\infty}i^2x^{i-1} \ = \ \frac{(1-x)^2+2(1-x)x}{(1-x)^4} \ = \ \frac{1-x^2}{(1-x)^4} \ = \ \frac{1+x}{(1-x)^3}.
    \end{equation*}
    Then, multiplying both terms by $x$ yields
    \begin{equation*}
        \sum_{i=0}^{\infty}i^2x^{i} \ = \ \frac{x(1+x)}{(1-x)^2}, \qquad |x| < 1.
    \end{equation*}
    Since the $i=0$ term contributes 0, we may drop it. So, substituting back $a^b=x$, we have the desired result.
    \end{proof}

\section{Equivalency}
    We have found the moments of $Y_n^{s}$ using two different methods, which led to a recurrence in Section \ref{s:3} and an explicit formula in Section \ref{s:4}. These two approaches could be proved to be equivalent, but we limit our study to proving the equivalence of the formulas for the variance and the expected value. In order to do that, we need the convolution identity, a well-known result that can be found in \cite{Mil17}, and the Cauchy-Hadamard property of power series, which can be found for example in \cite[p.69]{Rudin_2010}.
    \begin{lemma}[Convolution identity]\label{lemma:conv_id}
    Let $f(z)$ and $g(z)$ be two exponential generating functions:
    \begin{equation*}
        f(z) \ = \ \sum_{n=0}^{\infty} a_n \frac{z^n}{n!},\qquad g(z) \  =  \ \sum_{n=0}^{\infty} b_n \frac{z^n}{n!}.
    \end{equation*}
    Their product is given by
    \begin{equation*}
        f(z)g(z) \ = \ \sum_{n=0}^{\infty}\left[\sum_{j=0}^n \binom{n}{j} a_j b_{n-j}\right] \frac{z^n}{n!}.
    \end{equation*}
    \end{lemma}
    \begin{lemma}[Cachy-Hadamard]\label{lemma:CH}
        Let \( \sum_{n=0}^\infty a_n z^n \) be a power series with real coefficients. If
        \begin{equation*}
            \lim_{n \to \infty} \sqrt[n]{|a_n|}=0,
        \end{equation*}
        then the series converges uniformly for any $z\in\mathbb{R}$.
    \end{lemma}
    The following result and its proof closely follow the content of \cite{SR90}.
    \begin{theorem}\label{th:eq1}
        Let $s = 1/p \geq1$ and $0<n\leq s$. The expected value defined by Equation \eqref{eq:EvnonRec}
        \begin{equation*}
            \mathbb{E}(Y_n^{s}) \ = \ \sum_{k=1}^n\binom{n}{k}(-1)^{k+1}\frac{1}{1-(1-p)^k}  
        \end{equation*}
        is the solution of the recurrence relation 
        \begin{equation*}
        \mathbb{E}(Y_n^{s}) \ = \ \frac{1+\sum_{i=1}^{n-1} \binom{n}{i} p^i (1 - p)^{n - i} \mathbb{E}({Y_{n-i}^{s}})}{1-(1 - p)^{n}}
        \end{equation*}
        defined for every $n>1$ with $\mathbb{E}(Y_1^{s})=1/p$.
    \end{theorem}
    \begin{proof}
    For simpler notation, let $q = (1 - p)$ and define $x_n = \mathbb{E}(Y_n^{s})$ for $n \geq 1$, with $x_0 := 0$.\\
    For $n = 1$, we have $x_1 = 1/p$, and the theorem holds in this case. This also includes the case $s = 1$. \\ ~\\
    For $n > 1$ and $s>1$, we can rewrite the recurrence as
        \begin{equation} \label{eq:recEqEv}
            x_n \ = \ 1 + \sum_{i=0}^n \binom{n}{i} q^{n-i} p^ix_{n-i}  \ = \ 1 + \sum_{j=0}^n \binom{n}{j} q^j p^{n-j}x_j.
        \end{equation}
        Therefore, the exponential generating function $F(z) := \sum_{n=0}^{\infty} x_n \frac{z^n}{n!}$ becomes
        \begin{equation*}
            \sum_{n=0}^{\infty} x_n \frac{z^n}{n!} \ = \ \sum_{n=0}^{\infty} \frac{z^n}{n!} \left[1 + \sum_{j=0}^n \binom{n}{j} q^j p^{n-j}x_j\right] - 1.
        \end{equation*}
        Note that we subtracted $1$ to correct for the term $n = 0$, since $x_0$ does not satisfy Equation \ref{eq:recEqEv}.
        Using Lemma \ref{lemma:conv_id}, with $a_j=q^{j} x_j$ and $b_{n-j}=p^{n-j}$, we have
        \begin{equation*}
            F(z) \ = \ e^z + F(qz)e^{(1-q)z} - 1.
        \end{equation*}
        Then, defining $H(z):=F(z)e^{-z}$, we have
        \begin{equation*}
            H(z) \ = \ H(qz) + (1 - e^{-z}).
        \end{equation*}
        This equation is recursive, since \( H(qz) \) can be expressed in terms of \( H(q^2z) \), \( H(q^3z) \), and so on. In other words,
        \begin{equation*}
            H(z) = H(q^mz) + \sum_{k=0}^{m-1} (1 - e^{-q^kz}) .
        \end{equation*}
        Expanding the first term on the right-hand side:
        \begin{equation*}
            H(q^mz) \ = \ F(q^mz)e^{-q^mz} \ = \ e^{-q^mz}\sum_{n=0}^{\infty} x_n \frac{(q^mz)^n}{n!}.
        \end{equation*}
        Now assume that $x_n\leq n/p$. Under this assumption, the series above is bounded by the power series $\sum_{n=0}^{\infty} \frac{n}{p}\frac{(q^mz)^n}{n!}$, which converges uniformly for all $z\in\mathbb{R}$ by Lemma \ref{lemma:CH}, since
        \begin{equation*}
            \lim_{n\to\infty}\sqrt[n]{a_n}\ =\ \lim_{n\to\infty}\sqrt[n]{\frac{n}{p \cdot n!}}\ =\ 0.
        \end{equation*}
       The uniform convergence of the series, combined with the continuity of its partial sums, ensures the continuity of the full series. Therefore, taking the limit as $m \to \infty$, and noting that $q \in (0,1)$ implies $q^m z \to 0$, we obtain
        \begin{equation*}
            H(z) \ = \ H(0) + \sum_{k=0}^{\infty} (1 - e^{-q^kz}) \ = \ \sum_{k=0}^{\infty} (1 - e^{-q^kz})\implies F(z) \ =\ e^z\sum_{k=0}^{\infty} (1 - e^{-q^kz}).
        \end{equation*}
        The series converges for any $q\in (0,1)$, since $(1 - e^{-q^k z}) \sim q^k z$ as $k \to \infty$ and $z\sum q^k$ is a convergent geometric series. Using the definition of $F(z)$ and the properties of the exponential function, we have
        \begin{equation*}
            \sum_{n=0}^{\infty}\frac{z^n}{n!}x_n\ \ = \ F(z) \ =\ \sum_{k=0}^{\infty}\sum_{n=0}^{\infty}\frac{z^n}{n!}\left(1-(1-q^k)^n\right) \ = \ \sum_{n=0}^{\infty}\frac{z^n}{n!}\sum_{k=0}^{\infty}\left(1-(1-q^k)^n\right),
        \end{equation*}
        where Tonelli's Theorem, found for example in \cite[p.67]{Protter}, justifies swapping the sums. We observe that the resulting expression is solved by
        \begin{equation*}
        x_n \ = \ \sum_{k=0}^{\infty}\left(1-(1-q^k)^n\right) \ = \ \sum_{k=0}^{\infty}\left[1-\sum_{j=0}^{n}\binom{n}{j}(-1)^{j}(q^k)^j\right] \ = \ \sum_{k=0}^{\infty}\sum_{j=1}^{n}\binom{n}{j}(-1)^{j+1}(q^k)^j.
        \end{equation*}
        Inverting the order of the sums and using the properties of the geometric series, we have the desired result
        \begin{equation*}
        x_n \ = \ \sum_{j=1}^{n}\binom{n}{j}(-1)^{j+1}\left(\sum_{k=0}^{\infty}(q^j)^k\right)\ = \ \sum_{j=1}^{n}\binom{n}{j}(-1)^{j+1}\frac{1}{1-q^j}.
        \end{equation*}
        We now need to verify whether $x_n$, as defined in the previous equation, indeed satisfies $x_n \leq n/p$, in accordance with the assumption made earlier. This bound is established in Proposition~\ref{prop:EVBound1}, thereby confirming that the solution we derived is valid.
    \end{proof}
    
    \begin{theorem}\label{th:eq2} 
        Let $s = 1/p \geq1$ and $0<n\leq s$. The second moment (see Equation \eqref{eq:EV2NonRec})
        \begin{equation*}
        \mathbb{E}\left(\left({Y_n^s}\right)^2\right)
        \ = \ \sum_{k=1}^n\binom{n}{k}(-1)^{k+1}\frac{1+(1-p)^k}{(1-(1-p)^k)^2}
        \end{equation*}
        is the solution of the recurrence relation 
        \begin{equation*}
            \mathbb{E}\left(\left({Y_n^{s}}\right)^2\right) \ = \ \frac{\sum_{i=1}^{n-1}\binom{n}{i} p^i (1 - p)^{n - i}\mathbb{E}(\left({Y_{n-i}^{s}}\right)^2)-1 + 2\mathbb{E}(Y_n^{s})}{1- (1 - p)^{n}} 
        \end{equation*}
        defined for every $n>1$ with $\mathbb{E}({Y_1^{s}}^2)=(2-p)/p^2$.
    \end{theorem}
    \begin{proof}
        This proof closely follows the approach used in the proof of Theorem \ref{th:eq1}. We adopt the same notation. Let $y_n = \mathbb{E}\left(\left({Y_n^{s}}\right)^2\right)$ for $n\geq1$ and define $y_0:=0$.\\
        For $n = 1$, we have $y_1 = (2-p)/p^2$, and the theorem holds in this case. This also includes the case $s = 1$.\\
        For $n > 1$ and $s > 1$, we can rewrite the recurrence as
        \begin{equation*}
            y_n \ = \ 1  + 2\sum_{i=0}^n \binom{n}{i} q^i p^{n-i}x_i + \sum_{i=0}^n \binom{n}{i} q^i p^{n-i}y_i.
        \end{equation*}
        Define the exponential generating function $G(z) := \sum_{n=0}^{\infty} y_n \frac{z^n}{n!}$ and its transformed counterpart $T(z) := G(z)e^{-z}$. Following the same steps as in the previous proof and using linearity, we derive the functional equation:
        \begin{equation*}
            T(z) \ = \ T(qz) + 2H(qz) + (1 - e^{-z})
        \end{equation*}
        where $H(z)$ is defined as in the previous proof.\\
        Assume now that $y_n\leq2n/p^2$. Under this assumption, we can apply Lemma \ref{lemma:CH} to show that the series defining $T(z)$ converges uniformly on $\mathbb{R}$, and hence $T(z)$ is well-defined and continuous for all $z \in \mathbb{R}$. Therefore, since $q \in (0,1)$, it follows that $T(q^m z) \to 0$ as $m \to \infty$. Iterating the previous equation and substituting the explicit form of $H(z)$, we get
        \begin{equation*}
            G(z) \ = \ e^z\sum_{k=0}^{\infty}\left[ 2H\left(q^{k+1}z\right)+ \left(1 - e^{-q^kz}\right)\right] \ = \ e^z\sum_{k=0}^{\infty} \left[2\sum_{l=0}^{\infty}\left(1-e^{-q^{(l+k+1)}z}\right)+\left(1 - e^{-q^kz}\right)\right].
        \end{equation*}
        We observe that there exist $t$ combinations of $(l,k)$ such that $(l+k+1)=t$, as we can have $k\in(0,1,\dots,t-1)$. Therefore we write
        \begin{equation*}
            G(z) \ = \ \sum_{t=0}^{\infty}(2t+1)(e^z-e^{(1-q^t)z})\implies y_n \ = \ \sum_{t=0}^{\infty}(2t+1)(1-(1-q^t)^n).
        \end{equation*}
        From the previous proof we know that
        \begin{equation*}
            y_n \ = \ \sum_{t=0}^{\infty}(2t+1)\sum_{j=1}^{n}\binom{n}{j}(-1)^{j+1}(q^t)^j \ = \ \sum_{j=1}^{n}\binom{n}{j}(-1)^{j+1}\sum_{t=0}^{\infty}(2t+1)(q^j)^t.
        \end{equation*}
        Using the geometric series identity and Lemma~\ref{lemma:series}, we have the desired result:
        \begin{equation*}
            y_n \ = \ \sum_{j=1}^{n}\binom{n}{j}(-1)^{j+1}\left(\frac{2q^j}{(1-q^j)^2}+\frac{1}{1-q^j}\right) \ = \ \ \sum_{j=1}^n\binom{n}{j}(-1)^{j+1}\frac{1+q^j}{(1-q^j)^2}.
        \end{equation*}
        We now need to verify whether $y_n$, as defined in the previous equation, indeed satisfies $y_n \leq 2n/p^2$, in accordance with the assumption made earlier. This bound is established in Proposition~\ref{prop:VarBound1}, thereby confirming that the solution we derived is valid.
    \end{proof}

\section{Elementary Bounds}
In this appendix we present elementary bounds, mainly used as a benchmark to evaluate the bounds in the main corpse of the paper.

\begin{proposition} \label{prop:EVBound1}
Assume $s\geq n>0$, $\ {Z_i^s}\overset{\mathrm{iid}}{\sim}\mathrm{Geom}(1/s)$, $\ Y_n^{s}=\max_{i\leq n}{Z_i^s}$. Then
\begin{equation} \label{eq:EVBound1}
    n \ \leq \  s \ \leq \ \mathbb{E}(Y_n^{s}) \ \leq \ n s.  
\end{equation}
\end{proposition}

\begin{proof}
    We have \(n \leq s\) (by assumption). Since \(Y_n^{s} \ = \ \max_{i\leq n} {Z_i^s} \ \geq \ Z_1\), it follows that \(\mathbb{E}(Y_n^{s}) \ \geq \ \mathbb{E}(Z_1) \ = \ s\) by the monotonicity of the expected value. Also, since \(\mathbb{E}({Z_i^s}) \ = \ s \ > \ 0\), we have
\begin{equation*}
\mathbb{E}(Y_n^{s}) \ = \ \mathbb{E}\left(\max_{i\leq n} {Z_i^s}\right) \ \leq \ \mathbb{E}\left( \sum_{i=1}^n {Z_i^s} \right) \ = \ \sum_{i=1}^n \mathbb{E}({Z_i^s}) \ = \ ns
\end{equation*}
by the linearity of the expected value.
\end{proof}

It follows that $\mathbb{E}(Y_n^{s})=O(n)$ for fixed values of $s$ and $\mathbb{E}(Y_n^{s})=O(s)$ for fixed values of $n$. 

\begin{proposition} \label{prop:VarBound1}
    Assume $s \ \geq \ n \ > \ 0$, $\ {Z_i^s}\overset{\mathrm{iid}}{\sim}\mathrm{Geom}(1/s)$, $\ Y_n^{s}\ = \ \max_{i\leq n}{Z_i^s}$. Then
    \begin{equation}
    s^2\left(2-\frac{1}{s}\right) \ \leq \ \mathbb{E}\left(\left({Y_n^{s}}\right)^2\right) \ \leq \ ns^2\left(2-\frac{1}{s}\right),
    \end{equation}
    and therefore
    \begin{equation} \label{eq:VarBound1}
        s^2-s \ \leq \ \mathrm{Var}(Y_n^{s}) \ \leq \ s^2(2n-1)-ns.
    \end{equation}
\end{proposition}

\begin{proof}
    Since ${Z_i^s} > 0$, then $\left({Y_n^{s}}\right)^2 \ = \ (\max_{i \leq n} {Z_i^s})^2 \ = \ \max_{i \leq n} \left( Z_i^s \right) ^2$.\\We observe that
    \begin{equation*}
        \left(Z_1^s\right)^2 \ \leq \ \max_{i \leq n} \left( Z_i^s \right) ^2 \ \leq \ \sum_{i=1}^n \left( Z_i^s \right) ^2.
    \end{equation*}
    Since each ${Z_i^s} \sim \mathrm{Geom}(1/s)$, its second moment can be found in \cite[p.97]{casella2002statistical} and it is
    \begin{equation*}
      \mathbb{E}\left(\left( Z_i^s \right) ^2\right) \ = \ \frac{2 - \frac{1}{s}}{\left(\frac{1}{s}\right)^2} \ = \ s^2\left(2 - \frac{1}{s} \right).
    \end{equation*}
    Therefore, by the monotonicity and linearity of the expected value, it follows that
    \begin{equation*}
      s^2\left(2 - \frac{1}{s} \right) \ = \ \mathbb{E}\left(\left( Z_1^s \right) ^2\right) \ \leq \ \mathbb{E}\left(\max_{i\leq n}\left( Z_i^s\right) ^2\right) \ \leq \ \sum_{i=1}^n \mathbb{E}\left(\left( Z_i^s \right) ^2\right) \ = \ ns^2\left(2 - \frac{1}{s} \right).
    \end{equation*}
\end{proof}

It follows that $\mathrm{Var}(Y_n^{s}) = O(s^2)$ (as $s \to \infty$, for fixed values of $n$). \\

\vspace{\baselineskip}

\printbibliography

\vspace{\baselineskip}

\end{document}